\documentclass[12pt]{amsart}
\input epsf

\usepackage{amsfonts,amsthm,amsmath,amssymb,latexsym}
\usepackage{graphicx,color}
\usepackage[all]{xy}
\usepackage{wrapfig}

\address{HH:Department of Mathematics, Department of Mathematics, Kansas State University. Manhattan, KS 66502}
\email{hakobyan@math.ksu.edu}

\address{DS:Department of Mathematics, Queens College of CUNY,
65-30 Kissena Blvd., Flushing, NY 11367}
\email{Dragomir.Saric@qc.cuny.edu}

\address{DS:Mathematics PhD. Program, The CUNY Graduate Center, 365 Fifth Avenue, New York, NY 10016-4309}

\theoremstyle{definition}

 \newtheorem{definition}{Definition}[section]
 \newtheorem{remark}[definition]{Remark}

\theoremstyle{plain}

 \newtheorem{theorem}[definition]{Theorem}
 \newtheorem{corollary}[definition]{Corollary}
 \newtheorem{lemma}[definition]{Lemma}

\newcommand{\eps}{\varepsilon}
\newcommand{\G}{\Gamma}
\newcommand{\g}{\gamma}
\newcommand{\D}{\Delta}

\renewcommand{\d}{\delta}
\newcommand{\dist}{\mathrm{dist}}
\newcommand{\m}{\mathrm{mod}}
\renewcommand{\O}{\Omega}
\newcommand{\diam}{\mathrm{diam}}

\title[Visual sphere of the Universal Teichm\"uller space]{Visual sphere and Thurston's boundary of the Universal Teichm\"uller space}
\author{Hrant Hakobyan and Dragomir \v Sari\' c}
\thanks{The second author was partially supported by National Science Foundation grant DMS 1102440 and by the Simons Foundation grant.}

\begin{document}


\date{\today}

\begin{abstract}
Thurston's boundary to the universal Teichm\"uller space $T(\mathbb{D})$ is the space  $PML_{bdd}(\mathbb{D})$  of projective
bounded measured laminations of $\mathbb{D}$. 
A geodesic ray in $T(\mathbb{D})$ is of Teichm\"uller type if it shrinks vertical foliation of an integrable holomorphic quadratic differential. In a prior work we established that each Teichm\"uller geodesic ray limits to a multiple (by the reciprocal of the length of the leaves) of vertical foliation of the quadratic differential. 

Certain non-integrable holomorphic quadratic differential induce geodesic rays and we consider their limit points in $PML_{bdd}(\mathbb{D})$. Somewhat surprisingly, the support of the limiting projective measured laminations might be a geodesic lamination whose leaves are not homotopic to leaves of either vertical or horizontal foliation of the non-integrable holomorphic quadratic differential.
\end{abstract}

\maketitle

\section{Introduction}

Let $\mathbb{D}$ be the unit disk model of the hyperbolic plane.
The Teichm\"uller space $T(\mathbb{D})$ of the hyperbolic plane $\mathbb{D}$, called
the {\it universal} Teichm\" uller space, consists of all quasisymmetric maps
$h:S^1\to S^1$ which fix $1$, $i$ and $-1$ (cf. \cite{GL}). The
Teichm\" uller space of an arbitrary hyperbolic surface embeds in
$T(\mathbb{D})$ as a complex Banach submanifold. Thurston's boundary to the universal Teichm\"uller space $T(\mathbb{D})$ is the space $PML_{bdd}(\mathbb{D})$ of projective bounded measured laminations  of $\mathbb{D}$ (cf. \cite{Sa2}, \cite{Sa3}). Teichm\" uller geodesic rays are obtained by shrinking vertical trajectories of integrable holomorphic quadratic differentials. A Teichm\" uller geodesic ray 
corresponding to an integrable holomorphic quadratic differential $\varphi$ limits to a unique point in Thurston's boundary whose support geodesic lamination is homotopic to vertical foliation of $\varphi$ and the transverse measure is given by integrating the reciprocal of the lengths of vertical leaves against $Re(\sqrt{\varphi dz^2})$ (cf. \cite{HaSar1}).  
Certain non-integrable holomorphic quadratic differentials induce geodesic rays in $T(\mathbb{D})$ by shrinking their vertical trajectories in the same fashion as for integrable differentials. We study the limits of these geodesic rays on Thurston's boundary to $T(\mathbb{D})$.

\vskip .2 cm

The space $G(\mathbb{D})$ of oriented geodesics of $\mathbb{D}$ is identified with $S^1\times S^1-diag$ since each geodesic is uniquely determined by the ordered pair of its ideal endpoints on $S^1$. A {\it geodesic current} is a positive Borel measure on $G(\mathbb{D})$. The universal Teichm\"uller space $T(\mathbb{D})$ embeds into the space of geodesic currents when equipped with the uniform weak* topology (cf. \cite{Sa3}). Thurston's boundary to $T(\mathbb{D})$ is the set of asymptotic rays to the image of $T(\mathbb{D})$ in the space of geodesic currents and it is identified with the space $PML_{bdd}(\mathbb{D})$ of projective bounded measured laminations of $\mathbb{D}$ (cf. \cite{Sa3}). This approach was first introduced by
Bonahon
\cite{Bon1} to give an alternative description of Thurston's boundary of the Teichm\"uller space $T(S)$ of a closed surface $S$ of genus at least two.

In the case of closed surfaces, Masur \cite{Mas} proved that 
Teichm\"uller geodesic rays obtained by shrinking vertical trajectories of
holomorphic quadratic differentials with uniquely ergodic vertical foliations
converge to the projective classes of their vertical foliations in 
Thurston's boundary. However, when vertical foliations of holomorphic quadratic differentials on closed surfaces are not uniquely ergodic then the limit sets of the corresponding Teichm\"uller geodesic rays consist of more than one point while their supports are homotopic to vertical foliation of the quadratic differential (cf. \cite{Len}, \cite{LLR}).
On the other hand, the limits of Teichm\"uller geodesic rays in the universal Teichm\"uller space $T(\mathbb{D})$ corresponding to integrable holomorphic quadratic differentials always have a unique endpoint in Thurston's boundary of $T(\mathbb{D})$ (cf. \cite{HaSar1}).

Let $\varphi$ be an integrable holomorphic quadratic differential on $\mathbb{D}$. 
 Each vertical trajectory
 of $\varphi$ has two distinct endpoints on the boundary circle $S^1$ of the hyperbolic plane $\mathbb{D}$ (cf. \cite{Str}).
 Thus each vertical trajectory of $\varphi$ is homotopic to a unique geodesic of $\mathbb{D}$ relative ideal endpoints on $S^1$. Let $v_{\varphi}$ be the set of the geodesics in $\mathbb{D}$ homotopic to the vertical trajectories of $\varphi$.
Given a box of geodesics $[a,b]\times [c,d]\subset S^1\times S^1-diag$,
denote by $I_{[a,b]\times [c,d]}$ (at most countable) union of sub-arcs of horizontal trajectories that
intersects exactly once each vertical trajectory of $\varphi$ with one endpoint in $[a,b]$
and the other endpoint in $[c,d]$, and that does not intersect any other vertical trajectories of $\varphi$.

Define measured laminations $\nu_{\varphi}$ and $\mu_{\varphi}$ of $\mathbb{D}$ supported on $v_{\varphi}$ by
$$
 \nu_{\varphi}([a,b]\times [c,d])=\int_{I_{[a,b]\times [c,d]}}dx
$$
and
$$
 \mu_{\varphi}([a,b]\times [c,d])=\int_{I_{[a,b]\times [c,d]}}\frac{1}{l(x)}dx
$$
where $x=\int_{*}\sqrt{\varphi}dz$ is the natural parameter of $\varphi$ and $l(x)$ is the $\varphi$-length of the vertical trajectory through $x$ (cf. \cite{HaSar1}). Then (cf. \cite{HaSar1})
$$
\epsilon T_{\epsilon}\to \mu_{\varphi}
$$
as $\epsilon\to 0^{+}$ in the weak* topology on geodesic currents, where $T_{\epsilon}$ is a quasiconformal map of $\mathbb{D}$ that shrinks the vertical trajectories of $\varphi$ by a multiplicative constant $\epsilon$. In other words, the Teichm\"uller geodesic ray $T_{\epsilon}$ converges to $[\mu_{\varphi}]\in PML_{bdd}(\mathbb{D})$.

The space of all geodesic rays in the Teichm\"uller metric starting at the basepoint $[id]\in T(\mathbb{D})$ leaving every bounded subset of $T(\mathbb{D})$ is called the {\it visual boundary} of the universal Teichm\"uller space $T(\mathbb{D})$. The {\it Teichm\"uller geodesic rays}-obtained by shrinking the vertical direction of an integrable holomorphic quadratic differential $\varphi$-form an open and dense subset of $T(\mathbb{D})$ (cf. \cite{GL}).
However, there exist geodesic rays different from Teichm\"uller geodesic rays. A Beltrami coefficient $\xi$ of a quasiconformal map $f:\mathbb{D}\to\mathbb{D}$ is said to be {\it extremal} if $\|\xi\|_{\infty}$ is minimal among all Beltrami coefficients of quasiconformal maps representing the same point in $T(\mathbb{D})$ (where $f,g:\mathbb{D}\to\mathbb{D}$ represent the same point of $T(\mathbb{D})$ if $f|_{S^1}=g|_{S^1}$ \cite{GL}). If an extremal Beltrami coefficient $\xi$ is not of the Teichm\"uller type $k\frac{|\varphi |}{\varphi}$ for $0<k<1$ and $\varphi$ integrable, then $t\mapsto t\xi$ for $t\in [0,\frac{1}{\|\xi\|_{\infty}})$ is a geodesic ray that is not a Teichm\"uller geodesic ray.

We consider the limits of two (non-Teichm\"uller) geodesic rays introduced by Strebel \cite{GL}. The first example is given by a horizontal strip $S=\{ 0<Im(z)<1\}$ with the Beltrami coefficient $\xi =k\frac{|\varphi |}{\varphi}$ with $\varphi(z)\equiv 1$. Since $S$ does not have finite Euclidean area, the holomorphic quadratic differential $\varphi$ is not integrable and the corresponding geodesic ray is not Teichm\"uller. Note that $S$ is conformally identified with $\mathbb{D}$ and this identification is implicitly assumed. We denote by $T_{\epsilon}$ the shrinking of vertical trajectories by the factor $\epsilon$ and denote by $T_{1/\epsilon}$ the stretching of the vertical trajectories of $\varphi$ by the factor $1/\epsilon$ as $\epsilon\to 0^{+}$. We prove (cf. Theorem \ref{thm:strip} and Figure 1)

\vskip .2 cm

\noindent {\bf Theorem 1.} {\it 
Let $S=\{ 0< Im(z)<1\}$ be a horizontal strip and let $\varphi (z)=1$ for all $z\in S$.
Denote by $T_{\epsilon}$, $\epsilon >0$, the geodesic ray in $T(\mathbb{D})$ obtained by shrinking the vertical leaves of $\varphi$ by a factor $\epsilon$ and denote by $T_{\frac{1}{\epsilon}}$, $\epsilon >0$, the geodesic ray in $T(\mathbb{D})$ obtained by stretching the vertical leaves by a factor $\frac{1}{\epsilon}$.

Let $\nu_1$ be the (hyperbolic) measured lamination on $S$ whose support is
homotopic to the vertical foliation of $\varphi (z)=1$ on $S$ and whose transverse
measure is given by the euclidean length of the transverse horizontal set.
Let $\nu_2$ be the dirac measured lamination on $S$ with support the
hyperbolic geodesic homotopic to horizontal trajectories in $S$. 

Then we have
$$
T_{\epsilon}\to [\nu_1]
$$
and
$$
T_{1/\epsilon}\to [\nu_2]
$$
as $\epsilon\to 0^{+}$ in Thuston's boundary $ PML_{bdd}(\mathbb{D})$ of the universal Teichm\"uller space $T(\mathbb{D})$. The rate of convergence of $T_{\epsilon}$ is $1/\epsilon$ and the rate of convergence of $T_{1/\epsilon}$ is $1/\epsilon^{*}$, where $\epsilon^{*}\to 0^{+}$ as $\epsilon\to 0^{+}$.}

\vskip .2 cm

\noindent {\bf Remark 1.} Note that all vertical trajectories in $S$ have finite $\varphi$-lengths which is the same as in the case of integrable holomorphic quadratic differentials. On the other hand, horizontal trajectories of $\varphi$ have infinite lengths. Unlike for intergrable case, this makes the $\varphi$-metric unsuitable for making allowable metrics when computing moduli of various quadrilaterals and we find a new method for dealing with the difficulty.

\vskip .2 cm

Next we
consider Strebel's chimney domain $C=\{ z:Im(z)<0\}\cup \{ z:|Re(z)|<1\}$. The holomorphic quadratic differential $\varphi (z)dz^2=dz^2$ is not integrable on $C$ while the corresponding Beltrami coefficient $k\frac{|\varphi |}{\varphi}=k$ is extremal. Denote by $T_{\epsilon}$ as $\epsilon\to 0^{+}$ the geodesic ray obtained by shrinking the vertical foliation of $\varphi$ by the factor $\epsilon$. We prove (cf. Theorem \ref{thm:chimney} and Figure 2)

\vskip .2 cm

\noindent {\bf Theorem 2.} {\it Let $\nu$ be a measured lamination on $C$ which is a sum of two Dirac
measured laminations supported on geodesics $\g_1$ and $\g_2$ in $C$ with endpoints $1,+\infty\in \partial C$ and endpoints $-1,+\infty\in \partial C$, respectively. Then
$$
T_{\epsilon}\to [\nu ]
$$
as $\epsilon\to 0^{+}$ in Thuston's closure $T(\mathbb{D})\cup PML_{bdd}(\mathbb{D})$ of the universal Teichm\"uller space $T(\mathbb{D})$. The rate of convergence of $T_{\epsilon}$ is $1/\epsilon^{*}$, where $\epsilon^{*}\to 0^{+}$ as $\epsilon\to 0^{+}$.}

\vskip .2 cm

\noindent {\bf Remark 2.} All vertical leaves on $C$ have infinite lengths. If vertical leaves are straightened into hyperbolic geodesics, then the geodesic lamination $v_{\varphi}$ does not contain $g_1$ and $g_2$ even in its closure. Therefore it is impossible to detect $g_1$ and $g_2$ just by $v_{\varphi}$ alone. In fact, the limits $g_1$ and $g_2$ appear due to the fact that vertical trajectories accumulate to parts of the boundary of $C$.

\section{Thurston's boundary via geodesic currents}

We identify the hyperbolic plane with its upper half-plane model $\mathbb{D}$; the visual
boundary $S^1=\mathbb{R}\cup\{\infty\}$ to $\mathbb{D}$ is homeomorphic to the unit circle. An orientation preserving homeomorphism $h:S^1\to
S^1$ is said to be {\it quasisymmetric} if there exists $M\geq 1$ such that
$$
\frac{1}{M}\leq\frac{|h(e^{x+t})-h(e^x)|}{|h(e^x)-h(e^{x-t})|}\leq M
$$
for all $x\in\mathbb{R}$ and $t>0$. A
homeomorphism is quasisymmetric if and only if it extends to a quasiconformal
map of the unit disk.

\begin{definition}
The universal Teichm\"uller space $T(\mathbb{D})$ consists of all
quasisymmetric maps $h:S^1\to S^1$ that fix $1,i,-1\in S^1$.
\end{definition}

If $g:\mathbb{D}\to\mathbb{D}$ is a quasiconformal map, denote by $K(g)$ its
quasiconformal constant. The Teichm\"uller metric on $T(\mathbb{D})$ is given
by $d(h_1,h_2)=\inf_g K(g)$, where $g$ runs over all quasiconformal
extensions of the quasisymmetric map $h_1\circ h_2^{-1}$. The Teichm\"uller
topology is induced by the Teichm\"uller metric.

The space $G(\mathbb{D})$ of oriented geodesics on $\mathbb{D}$ is identified with $S^1\times S^1-diag$. A {\it geodesic current} is a Borel measure on $G(\mathbb{D})$.
 The {\it Liouville measure} $\mathcal{L}$ on the space of geodesic of $\mathbb{D}$ is given by
 $$
 \mathcal{L}(A)=\int_A \frac{|dx| |dy|}{|x-y|^2}
 $$
 for any Borel set $A\subset S^1\times S^1-diag$. If $A=[a,b]\times [c,d]$ is a {\it box of geodesic} then
 $$
 \mathcal{L}([a,b]\times [c,d])=\log\frac{(a-c)(b-d)}{(a-d)(b-c)}.
 $$

The universal Teichm\"uller space $T(\mathbb{D})$ maps into the space of geodesic currents
by taking the pull backs by quasisymmetric maps of the Liouville measure.
A geodesic current $\alpha$ is {\it bounded}
if
$$
\sup_{[a,b]\times [c,d]}\alpha ([a,b]\times [c,d])<\infty
$$
where the supremum is over all boxes of geodesics $[a,b]\times [c,d]$ with
$\mathcal{L}([a,b]\times [c,d])=\log 2$. The pull backs $h^{*}(\mathcal{L})$
for $h$ quasisymmetric are bounded geodesic currents (cf. \cite{Sa2}).

The pull backs of the Liouville measure define a homeomorphism
of $T(\mathbb{D})$ onto its image in the bounded geodesic currents, when the space of geodesic currents is equipped with the uniform weak* topology (\cite{Sa3}). The asymptotic rays to the image of $T(\mathbb{D})$  are
identified with the space of projective bounded measured laminations (cf.
\cite{Sa3}, \cite{Sa2}). Thus Thurston's boundary of $T(\mathbb{D})$ is the space
$PML_{bdd}(\mathbb{D})$ of all projective bounded measured laminations on
$\mathbb{D}$ (and an analogous statement holds for any hyperbolic Riemann
surface). Bonahon \cite{Bon1} introduced this approach for closed surfaces in order to give an alternative definition of Thurston's boundary.

\section{The asymptotics of the modulus}\label{section:background}

Let $(a,b,c,d)$ be a quadruple of distinct points on $S^1$ given in the
counterclockwise order. Denote by $\Gamma_{[a,b]\times [c,d]}$ the family of
all differentiable curves whose interiors are in $\mathbb{D}$ that have one
endpoint on the arc $[a,b]\subset S^1$ and the other endpoint on the arc
$[c,d]\in S^1$. An {\it admissible metric} $\rho$ for the family
$\Gamma_{[a,b]\times [c,d]}$ is a non-negative measurable function on
$\mathbb{D}$ such that the $\rho$-length of each
$\gamma\in\Gamma_{[a,b]\times [c,d]}$ is at least one, namely
$$
l_{\rho}(\gamma )=\int_{\gamma}\rho (z)|dz|\geq 1.
$$

The {\it modulus} $\m(\Gamma_{[a,b]\times [c,d]})$ of the family
$\Gamma_{[a,b]\times [c,d]}$ is given by
$$
\m(\Gamma_{[a,b]\times [c,d]})=\inf_{\rho}\int_{\mathbb{D}}\rho(z)^2dxdy
$$
where the infimum is over all admissible metrics $\rho$.

Lemma \ref{lemma:modproperties} below, summarizes some of the main properties
of the modulus, which we will use repeatedly throughout the paper. We refer
the reader to \cite{GM,LV,Vaisala:lectures} for the proofs of these
properties below and for further background on modulus.

If $\G_1$ and $\G_2$ are curve families in $\mathbb{C}$, we will say that
$\G_1$ \emph{overflows} $\G_2$ and will write $\G_1>\G_2$ if every curve
$\g_1\in \G_1$ contains some curve $\g_2\in \G_2$.

\begin{lemma} \label{lemma:modproperties}
Let $\G_1,\G_2,\ldots$ be curve families in $\mathbb{C}$. Then
\begin{itemize}
  \item[1.] \textsc{Monotonicity:} If $\G_1\subset\G_2$ then $\m(\G_1)\leq
      \m(\G_2)$.
  \item[2.] \textsc{Subadditivity:} $\m(\bigcup_{i=1}^{\infty} \G_i) \leq
      \sum_{i=1}^{\infty}\m(\G_i).$
  \item[3.] \textsc{Overflowing:} If $\G_1<\G_2$ then $\m \G_1 \geq \m
      \G_2$.
\end{itemize}
\end{lemma}

We will mostly be interested in estimating moduli of families of curves in a
domain $\Omega\subset\mathbb{C}$ connecting two subsets of the boundary of
$\O$. Thus, given $E,F\subset\partial\O$ we denote
\begin{align}
(E,F;\O)=\{\g:[0,1]\to\O : \g(0)\in E \mbox{ and } \g(1)\in F \}
\end{align}
the family of curves $\g$ starting in $E$ and terminating in $F$. With this
notation we have
$$\G_{[a,b]\times[c,d]} = ((a,b),(c,d);\mathbb{D}).$$

If the domain $\O$ is clear from the context, we will suppress it from the
notation and just write $\G_{E,F}$ instead of $(E,F;\O)$.

Heuristically modulus of $(E,F;\O )$ measures the amount of curves
connecting $E$ and $F$ in the $\O$. The more ``short" curves they are the
bigger the modulus is. This heuristic may be made precise using a notion of
relative distance $\D(E,F)$, which we define next.

Given two continua $E$ and $F$ in $\mathbb{C}$ we denote
\begin{align}
  \D(E,F) := \frac{\mathrm{dist}(E,F)}{\min\{\diam E, \diam F\}},
\end{align}
i.e. $\D(E,F)$ is the \textit{relative distance} between $E$ and $F$ in
$\mathbb{C}$.

\begin{lemma}[cf. \cite{HaSar1}]\label{lemma:mod-reldist}
For every pair of continua $E,F\subset\mathbb{C}$ we have
\begin{align}\label{modest:reldistance}
\m(E,F;\mathbb{C}) \leq \pi\left(1+\frac{1}{2\D(E,F)}\right)^2.
\end{align}
\end{lemma}

\begin{corollary} Let $E_n$ and $F_n$, $n\in\mathbb{N},$ be a sequence of pairs of continua in
$\mathbb{C}$.
  If  the sequence $\D(E_n,F_n)$ is bounded away from $0$ then  $\m(E_n,F_n;\mathbb{C})$ is bounded.
\end{corollary}
\begin{remark}
  The previous lemma is very weak for large $\D(E,F)$, since it is in fact
  easy to see that $\m(E,F,\mathbb{C})$ tends to $0$ as $\D(E,F)\to\infty$.
  But we will not need this estimate in the present paper and will refer the
  interested reader to Heinonen's book \cite{Hein} for relations between the
  modulus and relative distance.
\end{remark}
 The following lemma is an easy consequence of the asymptotic properties
of the moduli (cf. \cite{LV}).

\begin{lemma}[cf. \cite{HaSar}]
\label{lem:mod_liouville_measure} Let $(a,b,c,d)$ be a quadruple of points on
$S^1$ in the counterclockwise order. Let $\Gamma_{[a,b]\times [c,d]}$ consist
of all differentiable curves $\gamma$ in $\mathbb{D}$ which connect
$[a,b]\subset S^1$ with $[c,d]\subset S^1$. Then
$$
\m(\Gamma_{[a,b]\times [c,d]})-\frac{1}{\pi}\mathcal{L}([a,b]\times [c,d])-\frac{2}{\pi}\log 4\to 0
$$
as $\m(\Gamma_{[a,b]\times [c,d]})\to\infty$, where $\mathcal{L}$ is the
Liouville measure.
\end{lemma}

\begin{remark} Note that simultaneously $\m(\Gamma_{[a,b]\times [c,d]})\to\infty$ and $\mathcal{L}([a,b]\times [c,d])\to\infty$. Therefore it is enough to consider the asymptotic behaviour of the modulus in order to find the asymptotic behaviour of the Liouville measure.
\end{remark}

\section{The visual sphere of $T(\mathbb{D})$}

The visual sphere of the universal Teichm\"uller space $T(\mathbb{D})$, by definition,
consists of all {\it unbounded} geodesic rays for the Teichm\"uller metric starting at the basepoint $id\in
T(\mathbb{D})$. If a geodesic ray is a {\it Teichm\"uller ray} $t\mapsto
t\frac{|\varphi |}{\varphi}$ for $t\in [0,1)$ and
$\varphi$ integrable holomorphic quadratic differential, then the limit on Thurston's boundary equals to the projective class of $\mu_{\varphi}\in ML_{bdd}(\mathbb{D})$ (cf. \cite{HaSar1}). If $\eta$ is
an extremal Beltrami coefficient in its Teichm\"uller class, then  $t\mapsto
t\eta$ for $t\in [0,\frac{1}{\|\eta\|_{\infty}})$ defines a geodesic ray (cf.
\cite{GL}) and it corresponds to a single point on the visual sphere. An
interesting question is whether there exists a point on Thurston's boundary
to which the geodesic ray defined by an extremal Beltrami coefficient (not given in the Teichm\"uller form $k\frac{|\varphi |}{\varphi}$)
converges. We consider two examples of such geodesic rays both given by $t\mapsto
t\frac{|\varphi |}{\varphi}$ for $t\in [0,1)$, where $\varphi$ is a holomorphic quadratic differential that is not integrable on $\mathbb{D}$.

\subsection{The horizontal strip}

Consider a holomorphic quadratic differential $\varphi (z)dz^2=dz^2$ on
the horizontal strip $S=\{z:0<Im (z)<1\}$. Strebel (cf. \cite{Str}) proved
that the corresponding Beltrami coefficient $k\frac{|\varphi (z) |}{\varphi
(z)}=k$ is extremal. Note that $dz^2$ is not integrable since the euclidean
area of $S$ is infinite. We consider two geodesic rays: the shrinking
$T_{\epsilon}$ along the vertical foliation by the factor $\epsilon >0$ as
$\epsilon\to 0^{+}$ and the stretching $T_{1/\epsilon}$ along the vertical
foliation by the factor $\epsilon >0$ as $\epsilon\to 0^{+}$.

\begin{theorem}
\label{thm:strip}
Let $\nu_1$ be the (hyperbolic) measured lamination on $S$ whose support is
homotopic to the vertical foliation of $dz^2$ on $S$ and whose transverse
measure is given by the length in the natural parameter of the transverse horizontal set.
Let $\nu_2$ be the dirac measured lamination on $S$ with support the
hyperbolic geodesic homotopic to a horizontal trajectory in $S$. Then we have
$$
\epsilon (T_{\epsilon})^{*}(\mathcal{L})\to \nu_1
$$
and
$$
\epsilon^{*} (T_{1/\epsilon})^{*}(\mathcal{L})\to \nu_2
$$
where $\epsilon^{*}\to 0^{+}$ as $\epsilon\to 0^{+}$; $(T_{\epsilon})^{*}(\mathcal{L})$ is the pull-back of the Liouville geodesic current by the boundary map of $T_{\epsilon}$; similar for $(T_{1/\epsilon})^{*}(\mathcal{L})$. The convergence is in the weak* topology.
\end{theorem}

The first convergence follows directly from the considerations for integrable holomorphic quadratic differentials in \cite{HaSar} and \cite{HaSar1}.
It remains to prove the second convergence in the above Theorem. We note that stretching the vertical direction by $1/\epsilon$ is equivalent to shrinking the horizontal direction by $\epsilon$.

\begin{figure}
\noindent\makebox[\textwidth]{%
\includegraphics[width=12cm]{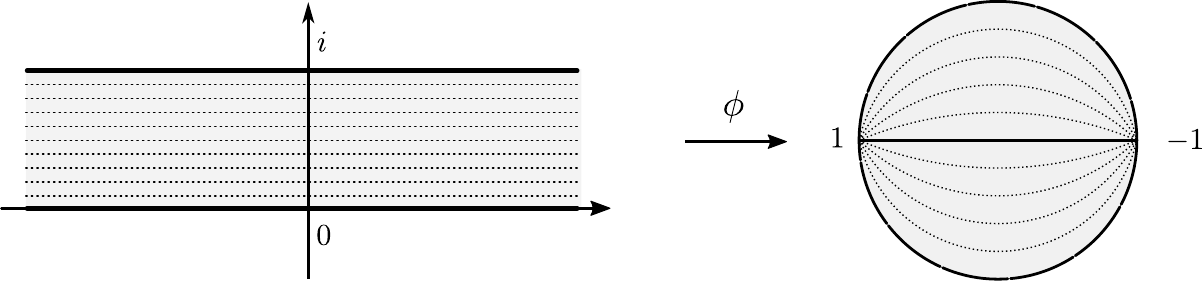}
}
\caption{Horizontal strip $S$ and its Riemann map. The dotted lines represent the horizontal trajectories of the standard quadratic differential $dz^2$ in $S$ and their images in $\mathbb{D}$ under $\phi$. }
\end{figure}%

For a pair of intervals of prime ends $I,J\subset \partial S$ we denote by
$\G_{I,J}$ the family of curves connecting $I$ and $J$ in the strip $S$, i.e.
$\G_{I,J}=(I,J; S)$ according to the notation used in Section
\ref{section:background}. Let $\G_{I,J}^{\eps} = H_{\eps}(\G_{I,J})$, where
$H_{\eps}(x,y) = (\eps x, y)$. Since $H(S)=S$ we have for $\eps>0$
\begin{align}
  \G_{I,J}^{\eps} = (H_{\eps}(I),H_{\eps}(J);S).
\end{align}

 We will denote by
$\phi$ the Riemann mapping from S to the unit disc $\mathbb{D}$. By
Caratheodory's theorem $\phi$ extends to $\partial{S}$ and we will denote
the extension by $\phi$ as well. Note that $\phi$  can be chosen to satisfy
the following properties for $x\in\mathbb{R}$:
\begin{eqnarray}
  \phi(0)&=&-i,\\
  \phi(i)&=&i,\\
  \phi(x+iy)&\to& \pm 1, \mbox{ as } x\to\pm\infty, y\in(0,1).
\end{eqnarray}

Let
\begin{eqnarray*}
I_0&=& (-\infty, 0) \cup (-\infty+i, i),\\
J_0&=& (1,\infty) \cup (1+i,\infty+i).
\end{eqnarray*}

Recall that a sequence of Borel measures $m_k$ on $S^1\times S^1-diag$ converges in the weak* topology to a Borel measure $m$ if for every box $[a,b]\times [c,d]$ with $m(\partial ([a,b]\times [c,d]))=0$ we have $m_k([a,b]\times [c,d])\to m([a,b]\times [c,d])$ as $k\to\infty$. Then
Theorem \ref{thm:strip} follows directly from the next lemma and the fact that $\m \G_{I_0,J_0}^{\eps}\to\infty$ as $\epsilon\to 0^{+}$ (cf. (\ref{eq:iojo})) by setting $\eps^{*}=1/\m \G_{I_0,J_0}^{\eps}$.

\begin{lemma}
If $I,J \subset \partial S$ are disjoint intervals of prime ends s.t.
$\m\G_{I,J}<\infty$ and the endpoints of $\phi(I)$ and $\phi(J)$ are disjoint
from $1$ and $-1$ then
\begin{equation}\label{horizontal}
  \lim_{\eps\to0}\frac{\m \G_{I,J}^{\eps}}{\m \G_{I_0,J_0}^{\eps}} =
  \begin{cases}
    1, & \mbox{\rm{ if }} -1\in\phi(I) \mbox{ \rm{and} } 1\in\phi(J),\\
   1, & \mbox{\rm{ if }} 1\in\phi(I) \mbox{ \rm{and} } -1\in\phi(J),\\
    0, & \mbox{\rm{ otherwise }}.
  \end{cases}
\end{equation}
\end{lemma}
\begin{proof} First we show that if $-1\in\phi(I)$ and $1\in\phi(J)$ then $
\m\G_{I,J}^{\eps}\to\infty$ as $\eps\to0$. For this let $I'\subset\mathbb{R}$
and $I''\subset(\mathbb{R}+i)$ be the two complementary intervals of the set
$I\cup J$ in $\partial{S}$. Then, since the curves connecting $I$ and $J$ are
exactly those which separate $I'$ from $I''$, it follows that
$$\m\G_{I',I''}^{\eps}=(\m\G_{I,J}^{\eps})^{-1},$$
and it is enough to show that $\m\G_{I',I''}^{\eps}\to0$ as $\eps\to0$. By
monotonicity of modulus we have
$$\m\G_{I',I''}^{\eps}=\m\G_{\eps I', \eps I''} \leq \m\G_{\eps
I',\mathbb{R}+i},$$ where $\eps I' = \{\eps x \in\mathbb{R} : x\in I'\}$ and
$\eps I'' = \{\eps x +i \in\mathbb{R}+i : x+i\in I''\}$. To show that the
last quantity tends to zero let $z_{\eps}$ be the center of the interval
$\eps I'\subset\mathbb{R}$ and denote by $\G_{\eps I'}$ the family of curves
connecting the boundary components of the annulus
$$\left\{z\in\mathbb{C} : \frac{\eps|I'|}{2}<|z-z_{\eps}|<1\right\}.$$
 Since $\G_{\eps I',\mathbb{R}+i}$ overflows $\G_{\eps I'}$, we have
$$\m\G_{\eps I',\mathbb{R}+i} \leq\m \G_{\eps I'} \leq
\frac{2\pi}{\log\frac{2}{\eps|I|}} \to 0, \mbox{ as } \eps\to0.$$ Therefore
$\m\G_{I,J}^{\eps}=(\m\G_{I',I''}^{\eps})^{-1}\to\infty$ and in particular
the denominator in (\ref{horizontal}) also tends to $\infty$.

\noindent{\textbf{Case 1}}: Suppose
\begin{align}\label{horizontal:case1}
-1\notin\phi(I)\cup\phi(J).
\end{align}
 We want
to show that in this case the limit in (\ref{horizontal}) is $0$. Since the
denominator of the quotient in (\ref{horizontal}) tends to $\infty$ it will
suffice to demonstrate that $\m\G_{I,J}^{\eps}$ stays bounded as $\eps\to0$.
We consider the following subcases:

\textit {Case ${1.1}$:} Suppose, in addition to (\ref{horizontal:case1}), we
also have
\begin{align}
1\notin\phi(I)\cup\phi(J).
\end{align}
In particular, we have $\max(\diam I, \diam J) <\infty$.  Thus, $I$ and $J$
are two bounded length intervals belonging either to the same boundary
component of $\partial{S}$ or to different components.

If $I$ and $J$ belong to the same component of $\partial S$ (assume this
 component is $\mathbb{R}$)  then considering the maps $F_{\eps}=(\eps^{-1}
 {id}){\circ H_{\eps}}$ we see that ${F_{\eps}}$ restricted to $\mathbb{R}$ is the
 indentity, and in particular $H_{\eps}(I)=I$ and $H_{\eps}(J)=J$. Therefore,
 by conformal invariance of $\eps^{-1}id$ we have,
 $$ \m \G_{I,J}^{\eps} = \m(H_{\eps}(I),H_{\eps}(J);H_{\eps}(S)) = \m(I,J;F_{\eps}(S))\leq \m
  (I,J;\mathbb{C}),$$ where, as before, $(I,J;\O)$ denotes the collection of curves
connecting $I$ and $J$ in the domain $\O$. Since $I$ and $J$ are bounded
fixed intervals a certain distance apart, we have that $\D(I,J)>0$ and
inequality (\ref{modest:reldistance}) implies that $\m(I,J;\mathbb{C})$ is
finite and therefore $\m \G_{I,J}^{\eps}$ is bounded for all $\eps>0$ and
(\ref{horizontal}) holds in this case.

%
%


\textsc{Case 1.2:} Suppose
\begin{align}
1\in \phi(I)\cup\phi(J).
\end{align}
Without loss of generality we may assume that
    $1\in\phi(J)$ and by (\ref{horizontal:case1}) then $I$ belongs to one of the components of $\partial{S}$, say $\mathbb{R}$, and  $\diam I<\infty$. By our normalization of $\phi$, this
    means that
    \begin{eqnarray*}
      I&=&(a,b),\\
      J&=&(c,\infty) \cup (d+i,\infty+i).
    \end{eqnarray*}
By subadditivity and monotonicity of modulus we have
\begin{eqnarray*}
  \m\G_{I,J}^{\eps}
  &\leq&
  \m\G_{I,(c,\infty)}^{\eps} + \m\G_{I,(d+i,\infty+i)}^{\eps}\\
  &\leq&
  \m\G_{I,(c,\infty)}^{\eps} + \m\G_{I,\mathbb{R}+i}^{\eps}.
\end{eqnarray*}
Just like in the beginning of the proof,
$\m\G_{I,\mathbb{R}+i}^{\eps}\leq c(\log\frac{2}{\eps\diam I})^{-1}\to0$.
Moreover, considering the maps $(\eps^{-1} {id}){\circ H_{\eps}}$  again, we
see that
$$ \m
\G_{I,(c,\infty)}^{\eps} \leq \m(H_{\eps}(I),H_{\eps}((c,\infty));H_{\eps}(S))\leq \m (I,(c,\infty);\mathbb{C}).$$

Since $\D(I,(c,\infty))>0$, we have  $\m (I,(c,\infty);\mathbb{C})<\infty$
and $\m\G_{I,(c,\infty)}^{\eps}$ and
$\m\G_{I,J}^{\eps}$ are bounded as $\eps\to0$ in this case as well.\\

\noindent\textbf{Case 2}: $\{-1,1\}\subseteq \phi(I) \cup \phi(J)$. Assume,
without loss of generality, that
$$-1\in\phi(I) \mbox{ and } 1\in\phi(J),$$
i.e. there are real numbers $a,b,c,d\in\mathbb{R}$ s.t.
\begin{eqnarray*}
I= (-\infty, a) \cup (-\infty+i, b+i) \mbox{ and }
J= (c,\infty) \cup (d+i,\infty+i).
\end{eqnarray*}
Note, that
$$-1\in\phi(I\cap I_0) \mbox{ and } 1\in\phi(J\cap J_0).$$
Therefore,
\begin{equation}
\label{eq:iojo}\lim_{\eps\to0}\m\G^{\eps}_{I\cap I_0,J\cap J_0}= \infty.
\end{equation}
By monotonicity and subadditivity of modulus we have
\begin{align}
  \begin{split}\m \G^{\eps}_{I\cap I_0,J\cap J_0} \leq  \m \G^{\eps}_{I,J}
  \leq& \, \m \G^{\eps}_{I\cap I_0,J\cap J_0}\\
  + & \, \m \G^{\eps}_{I\setminus I_0,J\cap J_0} + \m \G^{\eps}_{I, J\setminus J_0}, \label{modulusineq1}
  \end{split}\\
   \begin{split}\m \G^{\eps}_{I\cap I_0,J\cap J_0} \leq  \m \G^{\eps}_{I_0,J_0}
  \leq & \,\m \G^{\eps}_{I\cap I_0,J\cap J_0} \\
  + &\,\m \G^{\eps}_{I_0\setminus I,J_0\cap J} + \m \G^{\eps}_{I_0, J_0\setminus J} \label{modulusineq2}.
  \end{split}
\end{align}
Since $-1\in\phi(I\cap I_0)$, we may write $I\setminus I_0=I_1\cup I_2$ where
$I_1,I_2$ are (possibly empty) finite length intervals. By subadditivity, we
have
$$\m \G^{\eps}_{I\setminus I_0,J\cap J_0}\leq \m \G^{\eps}_{I_1,J_0\cap J}+\m \G^{\eps}_{I_1,J_0\cap J}.$$
Now, by Case 1.1 ($\diam I_i <\infty$ and $1\in\phi(J\cap J_0)$) we have that
$\m \G^{\eps}_{I_1,J_0\cap J}$ and $\m \G^{\eps}_{I_1,J_0\cap J}$ are both
bounded and therefore, so is $\m \G^{\eps}_{I\setminus I_0,J\cap J_0}$. Thus,
\begin{align*}
  \lim_{\eps\to0}\frac{\m \G^{\eps}_{I\setminus I_0,J\cap J_0}}{\m \G^{\eps}_{I\cap I_0,J\cap J_0}}=0
\end{align*}
%
The same argument also shows that
$$\lim_{\eps\to0}\frac{\m \G^{\eps}_{I,J\setminus J_0}}{\m \G^{\eps}_{I\cap I_0,J\cap J_0}}=0.$$
%
%
%
Therefore, dividing all the terms in (\ref{modulusineq1}) by $\m
\G^{\eps}_{I\cap I_0,J\cap J_0}$ and taking $\eps\to0$, results in
$$ \lim_{\eps\to0}\frac{\m \G^{\eps}_{I,J}}{\m \G^{\eps}_{I\cap I_0,J\cap J_0}}=1.$$
Similarly, (\ref{modulusineq2}) implies
$$ \lim_{\eps\to0}\frac{\m \G^{\eps}_{I_0,J_0}}{\m \G^{\eps}_{I\cap I_0,J\cap J_0}}=1,$$
and combining the last two equalities we obtain
$$ \lim_{\eps\to0}\frac{\m \G^{\eps}_{I,J}}{\m \G^{\eps}_{I_0,J_0}} =
\lim_{\eps\to0} \frac{\m \G^{\eps}_{I,J}}{\m \G^{\eps}_{I\cap I_0,J\cap J_0}} \left( \lim_{\eps\to0}\frac{\m \G^{\eps}_{I_0,J_0}}{\m \G^{\eps}_{I\cap I_0,J\cap J_0}} \right)^{-1}  =1,$$
as required.
%
%
\end{proof}


\subsection{The Strebel's chimney domain}

Let
$$C=\{ z:Im(z)<0\}\cup\{ z:|Re(z)|<1\}$$ be the Strebel's chimney domain (cf. \cite{GL}).
The holomorphic quadratic differential $\varphi (z)dz^2=dz^2$ is not
integrable on $C$. However, Strebel proved that the corresponding Beltrami
coefficient $k\frac{|\varphi |}{\varphi}$ is extremal.

We  denote by $\phi$ the Riemann mapping from $C$ to the unit disc
$\mathbb{D}$. By Caratheodory's theorem $\phi$ extends to the $\partial{C}$
and we will denote the extension by $\phi$ as well. Note that $\phi$  can be
chosen to satisfy the following properties for $z\in C$:
\begin{eqnarray*}
  \phi(\pm 1)&=&\pm1,\\
  \phi(z)&\to&+i, \mbox{ if } |z|\to\infty \mbox{ and } Im(z)>0,\\
  \phi(z)&\to&-i, \mbox{ if } |z|\to\infty \mbox{ and } Im(z)<0.
\end{eqnarray*}

\begin{figure}
\noindent\makebox[\textwidth]{%
\includegraphics[width=12cm]{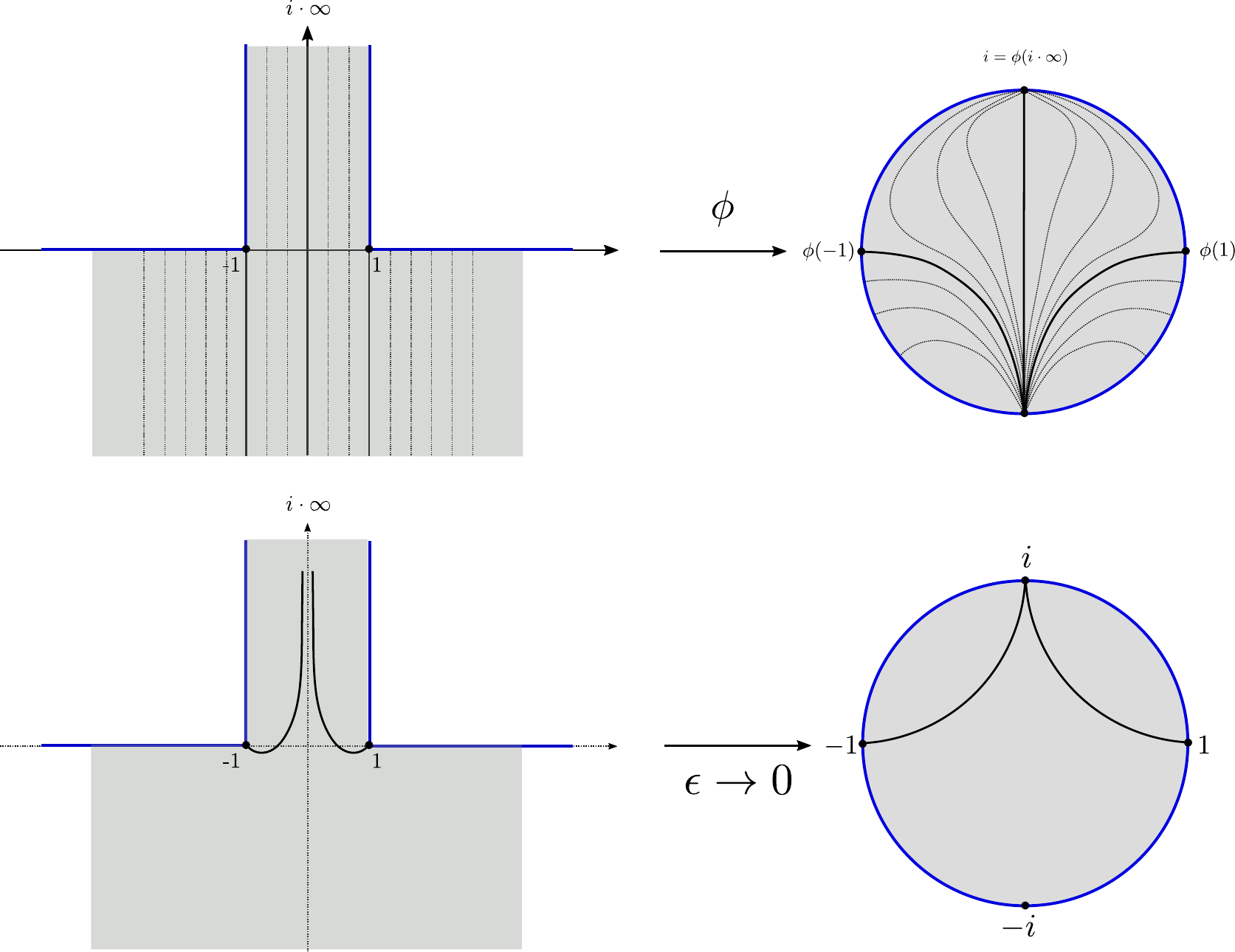}
}
\caption{Strebel's chimney domain and its Riemann map. The dotted lines represent the vertical trajectories of the standard quadratic differential $dz^2$ in $C$ and their images in $\mathbb{D}$ under $\phi$. }
\end{figure}%

\begin{theorem}
\label{thm:chimney}
Let $\nu$ be a measured lamination on $C$ which is a sum of two Dirac
measured laminations supported on geodesics $\g_1$ and $\g_2$ in $C$, where
$\phi(\g_1),\phi(\g_2)$ are the hyperbolic geodesics in $\mathbb{D}$
connecting $i$ to $-1$ and $1$, respectively. Then
$$
\epsilon^{*} (T_{\epsilon})^{*}(\mathcal{L})\to \nu
$$
where $\epsilon^{*}\to 0$ as $\epsilon\to 0^{+}$ and $T_{\epsilon}$ shrinks the vertical
trajectories by the factor $\epsilon$. As before, $(T_{\epsilon})^{*}(\mathcal{L})$ is the pull back of the Liouville current $\mathcal{L}$ and the convergence is in the weak* topology.
\end{theorem}

To prove this theorem we reformulate it in terms of the limiting values of
moduli of families of curves in $C$. Just like in the case of the strip, for
a pair of intervals of prime ends $I,J\subset\partial C$ we denote by
$\G_{I,J}$ the family of curves connecting $I$ and $J$ in the domain $C$ and
let $\G_{I,J}^{\eps} = T_{\eps}(\G_{I,J})$, where $T_{\eps}(x,y) = (x, \eps
y)$. Let us denote
\begin{align*}
I_0&= (1+i,1+i\cdot\infty) = \{(1,iy) : 1<y<\infty\} ,\\
J_0&= (1,2).
\end{align*}
Then we define $\epsilon^{*}=1/\m \G_{I_0,J_0}^{\eps}$ and we need to prove that $\epsilon^{*}\to 0$ as $\epsilon\to 0$.

\begin{theorem}
If $I,J \subset \partial C$ are disjoint intervals of prime ends s.t.
$\m\G_{I,J}<\infty$  then
\begin{equation}\label{chimney}
  \lim_{\eps\to0}\frac{\m \G_{I,J}^{\eps}}{\m \G_{I_0,J_0}^{\eps}} =
  \begin{cases}
    0, & \mbox{\rm{ if }} i\notin \phi(I),\\
    \#(\{-1,1\}\cap\phi(J)) & \mbox{\rm{ if }} i\in \phi(I).\\
  \end{cases}
\end{equation}
\end{theorem}

\begin{proof}
  First note that
  \begin{align}\label{modlim:standard}
\lim_{\eps\to0}\m\G_{I_0,J_0}^{\eps}=\infty.
  \end{align}
Indeed, letting $I'=[1,1+i)$ and $J'=\partial{C}\setminus(I_0\cup J_0 \cup
I')$ we obtain $\m\G^{\eps}_{I_0,J_0} = (\m\G^{\eps}_{I',J'})^{-1}.$ Since $\G^{\eps}_{I',J'}$ overflows the family of curves connecting the
boundary components of the annulus $A(1;\diam(T_{\eps}(I')), 1)$ centered at
$z=1$ with radii $\diam(T_{\eps}(I'))=\eps$ and $1$, we have
$$\m\G^{\eps}_{I',J'}\leq \frac{2\pi}{\log(\frac{1}{\eps})}\to0, \mbox{ as } \eps\to0,$$
and therefore $\m\G^{\eps}_{I_0,J_0}\to\infty$ as $\eps\to0$.

Now, let $C_{+},C_-$ denote the connected components of $\partial C$ in the
right and left half-planes, respectively. Furthermore, for an interval of
prime ends $I\subset\partial{C}$ we let $I_{\pm}=I\cap C_{\pm}$. Therefore
\begin{equation}\label{mod-spliting}
  \m\G_{I,J}^{\eps} \leq \m\G_{I_+,J_+}^{\eps}+\m\G_{I_-,J_-}^{\eps}+\m\G_{I_+,J_-}^{\eps}+\m\G_{I_-,J_+}^{\eps}.
\end{equation}
Below we will prove the following lemma.
\begin{lemma}\label{lemma:bdd-distinct}
  If $\m\G_{I,J}<\infty$ then $\m\G_{I_-,J_+}^{\eps}$ and
  $\m\G_{I_+,J_-}^{\eps}$ are bounded in $\eps$.
\end{lemma}

Therefore, Lemma \ref{lemma:bdd-distinct} and (\ref{modlim:standard}) imply
\begin{align}\label{modlim:chimney-same}
\lim_{\eps\to0}\frac{\m \G_{I,J}^{\eps}}{\m \G_{I_0,J_0}^{\eps}}\leq
\lim_{\eps\to0}\frac{\m \G_{I_+,J_+}^{\eps}}{\m \G_{I_0,J_0}^{\eps}}+
\lim_{\eps\to0}\frac{\m \G_{I_-,J_-}^{\eps}}{\m \G_{I_0,J_0}^{\eps}},
\end{align}
and we need estimates on $\m \G_{I_+,J_+}^{\eps}$ and $\m
\G_{I_-,J_-}^{\eps}$.  This is done in the next lemma.

\begin{lemma}\label{lemma:bdd-same} If $\m\G_{I,J}<\infty$ then
\begin{itemize}
  \item[(a)] If $i\notin\phi(I)\cup\phi(J)$ then $\m\G_{I_-,J_-}^{\eps}$
      and $\m\G_{I_+,J_+}^{\eps}$ are bounded in $\eps$.
  \item[(b)] If $i\in\phi(I)$ and $1\notin \bar{J_+}$ (resp. $-1\notin
      \bar{J_-}$) then $\m\G_{I_+,J_+}^{\eps}$ (resp.
      $\m\G_{I_-,J_-}^{\eps}$) is bounded in $\eps$.
  \item[(c)] If $i\in\phi(I)$ and $1\in J_+$ (resp. $-1\in J_-$) then
      \label{lemma:same}
\begin{equation}\label{limit:chimney}
\lim_{\eps\to0}\frac{\m\G_{I_+,J_+}^{\eps}}{\m\G_{I_0,J_0}^{\eps}}=1\qquad  \left(\mbox{ \textrm{resp.} }
\lim_{\eps\to0}\frac{\m\G_{I_-,J_-}^{\eps}}{\m\G_{I_0,J_0}^{\eps}}=1.\right)
  \end{equation}
\end{itemize}
\end{lemma}


Let us prove the theorem assuming Lemma \ref{lemma:bdd-same}.

 $-$ If $i\notin\phi(I)$ then (\ref{modlim:chimney-same}) and Lemma \ref{lemma:bdd-same}.(a)
 imply that  limit in (\ref{chimney}) is $0$.

 $-$ If $i\in\phi(I)$ and $\#(\{-1,1\}\cap\phi(J))=0$ then by
(\ref{modlim:chimney-same}) and Lemma \ref{lemma:bdd-same}.(b) the limit in
(\ref{chimney}) is $0$.
%
%

$-$ If $i\in \phi(I)\mbox{ and } \#(\{-1,1\}\cap\phi(J))=1$, we may assume
without loss of generality, that $1\in\phi(J) \mbox{ but } -1\notin\phi(J)$.
Then, by (\ref{modlim:chimney-same}) and Lemma \ref{lemma:bdd-same}.(b),(c)
we have
\begin{eqnarray*}
  \lim_{\eps\to0}\frac{\m \G_{I,J}^{\eps}}{\m \G_{I_0,J_0}^{\eps}}
  \leq 1+0.
\end{eqnarray*}
Since also $\m\G_{I,J}^{\eps} \geq \m \G_{I_+,J_+}^{\eps}$ we obtain that the
limit in (\ref{chimney}) is $1$.

$-$ If $\#(\{-1,1\}\cap\phi(J))=2$ then we have that in this case
\begin{align}\label{chimney:n=2}
i\in \phi(I)\mbox{ and } \{-1,1\}\subset\phi(J).
\end{align}
By (\ref{modlim:chimney-same}) and Lemma \ref{lemma:bdd-same}.(c) we
have
\begin{eqnarray*}
  \lim_{\eps\to0}\frac{\m \G_{I,J}^{\eps}}{\m \G_{I_0,J_0}^{\eps}}
  \leq 1+1=2.
\end{eqnarray*}
The rest of the proof is devoted to proving the opposite inequality. Note
that in the previous case this easily followed from the monotonicity of the
modulus. To estimate $\m \G_{I,J}^{\eps}$ from below, we will compare it to
$\m \G_{\tilde{I},\tilde{J}}^{\eps}$, where
\begin{eqnarray*}
\tilde{I}&:=&(-1+i,-1+i\infty)\cup(1+i,1+i\infty)\\
\tilde{J}&:=&(-2,-1)\cup(1,2).
\end{eqnarray*}

We will show that if $I$ and $J$ satisfy conditions (\ref{chimney:n=2}), then
the following equalities hold:
\begin{align}
\lim_{\eps\to0}\frac{\m \G_{I,J}^{\eps}}{\m
\G_{\tilde{I},\tilde{J}}^{\eps}}&=1, \label{modlim:chimney}\\
\lim_{\eps\to0}\frac{\m \G_{\tilde{I},\tilde{J}}^{\eps}}{\m
\G_{{I_0},{J_0}}^{\eps}}&=2. \label{modlim:chimney3}
\end{align}
This will be sufficient for the proof of the theorem in this case, since
multiplying (\ref{modlim:chimney}) and (\ref{modlim:chimney3}) clearly yields
equality (\ref{chimney}) in this case.

\begin{proof}[Proof of equality (\ref{modlim:chimney})]
By monotonicity and subadditivity of the modulus, we have
\begin{align}
 \begin{split}
 \m\G^{\eps}_{I\cap\tilde{I},J\cap\tilde{J}} \leq \m \G_{I,J}^{\eps}&
\leq \m\G^{\eps}_{I\cap\tilde{I},J\cap\tilde{J}}
+\m\G^{\eps}_{I\cap\tilde{I},J\setminus\tilde{J}}\\
&
+ \m\G^{\eps}_{I\setminus\tilde{I},J\cap\tilde{J}}+\m\G^{\eps}_{I\setminus\tilde{I},J\setminus\tilde{J}}. \label{ineq:case3}
\end{split}\\
\begin{split}
 \m\G^{\eps}_{I\cap\tilde{I},J\cap\tilde{J}} \leq \m \G_{\tilde{I},\tilde{J}}^{\eps}&
\leq \m\G^{\eps}_{I\cap\tilde{I},J\cap\tilde{J}}
+\m\G^{\eps}_{I\cap\tilde{I},\tilde{J}\setminus J} \\ &
+ \m\G^{\eps}_{I\setminus\tilde{I},J\cap\tilde{J}}+\m\G^{\eps}_{\tilde{I}\setminus I,\tilde{J}\setminus J}.\label{ineq:case3-1}
\end{split}
\end{align}
It is enough to show that all the terms on the right hand sides of
(\ref{ineq:case3}) and (\ref{ineq:case3-1}) are bounded, except
$\m\G^{\eps}_{I\cap\tilde{I},J\cap\tilde{J}}$. Indeed, if this is the case
then, since $\m \G_{I\cap \tilde{I},J\cap\tilde{J}}^{\eps}\to\infty$, we will
have
$$\lim_{\eps\to0}\frac{\m \G_{I,J}^{\eps}}{\m \G_{I\cap \tilde{I},J\cap\tilde{J}}^{\eps}}=1,
\quad\mbox{ and }\quad
\lim_{\eps\to0}\left(\frac{\m \G_{\tilde{I},\tilde{J}}^{\eps}}{\m \G_{I\cap \tilde{I},J\cap\tilde{J}}^{\eps}}\right)^{-1}=1.$$
Thus, multiplying the last two equations gives (\ref{modlim:chimney}).

Now we show the boundedness of the mentioned moduli appearing in
(\ref{ineq:case3}). The case of (\ref{ineq:case3-1}) is done in exactly the
same way.

Since $i\in\phi(I)\cap\phi(\tilde{I})$ we have that $I\setminus\tilde{I}$ is
a union of two (possibly empty) bounded segments. Therefore, subadditivity
and part $(a)$ of Lemma \ref{lemma:bdd-same} implies that
$\m\G^{\eps}_{I\setminus\tilde{I},J\cap\tilde{J}}$
 and $\m\G^{\eps}_{I\setminus\tilde{I},J\setminus\tilde{J}}$ are both
bounded.

To estimate $\m\G^{\eps}_{I\cap\tilde{I},J\setminus\tilde{J}}$ note, that
since $\{-1,1\}\subset\phi(J)$ it follows that $-i\in\phi(J)$ and therefore
 $$J\setminus \tilde{J}=(-\infty,-2]\cup J_1 \cup J_2 \cup[2,\infty),$$
where $J_1$ and $J_2$ are compact intervals in the vertical lines
$\{Re(z)=\pm1\}$, respectively. Therefore,
$$\m\G^{\eps}_{I\cap\tilde{I},J\setminus\tilde{J}}\leq
\m\G_{I\cap \tilde{I}, [2,\infty)}^{\eps}+ \m\G_{I\cap \tilde{I},
(-\infty,-2]}^{\eps} + \m\G_{I\cap \tilde{I}, J_1\cup J_2}^{\eps}.$$ Now,
$\m\G_{I\cap \tilde{I}, [2,\infty)}^{\eps}$ and $\m\G_{I\cap \tilde{I},
(-\infty,-2]}^{\eps}$ are bounded by Lemma \ref{lemma:bdd-distinct} and part
$(b)$ of Lemma \ref{lemma:bdd-same}. Moreover,  $\m\G_{I\cap \tilde{I},
J_1}^{\eps}$ and $\m\G_{I\cap \tilde{I}, J_2}^{\eps}$ are also both bounded,
since the relative distance between, say, $T_{\eps}(I\cap \tilde{I})$ and
$T_{\eps}(J_1)$ remains bounded away from $0$ as $\eps\to0$. It follows that
$\m\G_{I\cap \tilde{I}, J_1\cup J_2}^{\eps}$ is bounded and therefore
$\m\G^{\eps}_{I\cap\tilde{I},J\setminus\tilde{J}}$ is bounded as
well.
\end{proof}


\begin{proof}[Proof of equality (\ref{modlim:chimney3})]

We first compare $\m \G_{\tilde{I},\tilde{J}}^{\eps}$ and
$\m\G_{I_0,J_0}^{\eps}$. For this let
\begin{align*}
\G_{I_0,J_0}^{\eps,+}&:= \{\g\in\G_{I_0,J_0}^{\eps} : \g\subset \{ Re(z) > 0\} \},\\
\G_{I_0,J_0}^{\eps,0}&:= \{\g\in\G_{I_0,J_0}^{\eps} : \g\cap \{ Re(z) = 0\} \neq\emptyset \}.
\end{align*}
Note that, since $C$ and $\G_{\tilde{I},\tilde{J}}^{\eps}$ are both
symmetric with respect to the imaginary axis, the symmetry rule for modulus
(see \cite{GM}, page 137) implies that
\begin{align}\label{modeq:symmetry}
{\m\G_{\tilde{I},\tilde{J}}^{\eps}}=2 \,\m \G_{I_0,J_0}^{\eps,+}.
\end{align}

Moreover, by monotonicity and subadditivity of modulus we have
\begin{align}\label{modineq:symmetry}
\m\G_{I_0,J_0,}^{\eps,+}\leq \m\G_{I_0,J_0}^{\eps}\leq \m
\G_{I_0,J_0}^{\eps,+} + \m \G_{I_0,J_0}^{\eps,0}.
\end{align}
%
Since $T_{\eps}(I_0)\subset[1,1+i\infty)$ we have that
\begin{align}\label{modineq:cornertouching}
\m\G_{I_0,J_0}^{\eps,0}\leq \m \G_{[1,1+i\infty),[1,2]}^0,
\end{align}
where the latter is the family of curves connecting $[1,1+i\infty)$ to
$[1,2]$ in $C$ which also intersect the imaginary axis $\{x=0\}$.

It is easy to see that $\m \G_{[1,1+i\infty),[1,2]}^0<\infty$. Indeed,
letting $\G_1$ and $\G_2$ be the subfamilies of curves in
$\G_{[1,1+i\infty),[1,2]}^0$ starting in $[1,1+i]$ and $[1+i,1+i\infty)$,
respectively, we have
$$\m \G_{[1,1+i\infty),[1,2]}^0 \leq \m \G_1 + \m \G_2.$$
Now, note that $\m \G_1< \infty$, since $\G_1$ overflows the family of curves
connecting $[1,1+i]$ to $\{Re(z)=0\}$, which has finite modulus (relative
distance between $[1,1+i]$ and $\{Re(z)=0\}$ is $1>0$).
Moreover, $\m \G_2\leq 1$, since $\G_2$ overflows the family of curves
connecting the horizontal sides in the unit square $[0,1]\times[0,1]$.

Thus, $\m \G_{[1,1+i\infty),[1,2]}^0<\infty,$ and by
(\ref{modineq:cornertouching}) we also have that $\m \G_{I_0,J_0}^{\eps,0}$
is bounded independently of $\eps$.

Since $\m\G^{\eps}_{I_0,J_0}\to\infty$ and $\m \G_{I_0,J_0}^{\eps,0}$ is
bounded, it follows from (\ref{modineq:symmetry}) that
\begin{equation}\label{modlim:halving}
  \lim_{\eps\to0}\frac{\m \G_{I_0,J_0}^{\eps,+}}{\m\G_{I_0,J_0}^{\eps}} = 1.
\end{equation}
Finally, combining (\ref{modeq:symmetry}) and (\ref{modlim:halving}) we
conclude, that
\begin{equation}
 \lim_{\eps\to0}\frac{\m \G_{\tilde{I},\tilde{J}}^{\eps}}{\m \G_{{I}_0,{J}_0}^{\eps}}
 = \lim_{\eps\to0}\frac{\m \G_{\tilde{I},\tilde{J}}^{\eps}}{\m \G_{I_0,J_0}^{\eps,+}}
 \cdot \frac{\m \G_{I_0,J_0}^{\eps,+}}{\m\G_{{I}_0,{J}_0}^{\eps}} = 2 \cdot 1 =2,
\end{equation}
which proves (\ref{modlim:chimney3}).
\end{proof}

Thus, to complete the proof of the theorem we only need to prove Lemmas
\ref{lemma:bdd-distinct} and \ref{lemma:bdd-same}.
\end{proof}

\begin{proof}[Proof of Lemma \ref{lemma:bdd-distinct}] We will show the boundedness of $\m\G_{I_-,J_+}^{\eps}$. The
case of  $\m\G_{I_+,J_-}^{\eps}$ is done the same way.

There are two cases to consider:

\textsc{Case 1}. Suppose $\min(\diam I_-,\diam J_+)<\infty$. For concreteness
we may assume $\diam I_- <\infty$.  This means that there is a real number $1
< a <\infty$ such that for $\eps>0$ small enough we have
$$T_{\eps}(I_-)\subset(-a,-1]\cup[-1,-1+i).$$
Therefore, since
$$\D(T_{\eps}(I_-),T_{\eps}(J_+)) \geq \frac{2}{\diam (T_{\eps}(I_-))} \geq \frac{2}{\sqrt{1+a^2}},$$
by Lemma \ref{lemma:mod-reldist} we have
\begin{eqnarray*}
  \m\G_{I_-,J_+}^{\eps} = \m\G_{T_{\eps}(I_-),T_{\eps}(J_+)} \leq \m (T_{\eps}(I_-),T_{\eps}(J_+),\mathbb{C}) <\infty.
\end{eqnarray*}

\textsc{Case 2}. Suppose $\diam I_- = \diam J_+ =\infty$. Since
$\m\G_{I_-,J_+}\leq\m\G_{I,J}<\infty$ we have that
\begin{align*}
&\min\{\diam(I_-\cap(-\infty,-1]), \diam(J_+ \cap [1,\infty))\}<\infty,\\
&\min\{\diam(I_-\cap \{Re(z)=-1\}), \diam(J_+ \cap \{Re(z)=1\})\}<\infty.
\end{align*}
For concreteness we may assume then, that
\begin{align*}
\diam(I_-\cap(-\infty,-1])<\infty, \mbox{ and }
\diam(J_+ \cap \{Re(z)=1\})\}<\infty.
\end{align*}
Therefore there is a real number $1 < a <\infty$ such that for $\eps>0$ small
enough we have
$$T_{\eps}(I_-)\subset(-a,-1]\cup[-1,-1+i\infty) \mbox{ and }
T_{\eps}(J_+)\subset[1,\infty]\cup[1,1+i).$$
Therefore,
\begin{align}\label{modest:unbounded}
\begin{split}
  \m\G_{I_-,J_+}^{\eps}
  &\leq
  \m\G_{[-1,-1+i\infty),[1,\infty)} \\
  & +
  \m\G_{[-1,-1+i\infty),[1,1+i)}+
  \m\G_{(-a,-1],[1,\infty)} +
  \m\G_{(-a,-1],[1,1+i)},
\end{split}
\end{align}
where the last three terms are bounded by Case $1$ above. On the other hand,
$$\m\G_{[-1,-1+i\infty),[1,\infty)}\leq \m\G_{[-1,-1+i),[1,\infty)}+
\m\G_{[-1+i,-1+i\infty),[1,\infty)}.$$
Since $\D([-1,-1+i),[1,\infty))=2$ we have that the first term above is
bounded. Moreover,
$$\m\G_{[-1+i,-1+i\infty),[1,\infty)}<2,$$ since
$\G_{[-1,-1+i\infty),[1,\infty)}$ overflows the ``vertical family" of the
rectangle $[-1,1]\times[0,1]$. Therefore
$$\m\G_{[-1,-1+i\infty),[1,\infty)}<\infty,$$ and by inequality (\ref{modest:unbounded}) we have
that $\m\G^{\eps}_{I_-,J_+}$ is bounded.
%
%
%
%
\end{proof}

\begin{proof}[Proof of Lemma \ref{lemma:bdd-same}]
We will estimate only $\m\G_{I_+,J_+}$. The estimates for $\m\G_{I_-,J_-}$
are done in a very similar way.

\textsc{Case} ${(a)}:$ We first assume that $1\notin I_+\cup J_+$. Then we
have the following
subcases:\\
$(a1)$ If $I_+$ and $J_+$ belong to the same component of
      $C_+\setminus \{1\}$ then
      $$\D({T_{\eps}(I_+),T_{\eps}(J_+)}) = \D (I_+,J_+)>0.$$
Therefore $\m\G_{I_+,J_+}^{\eps}\leq \m
      ({T_{\eps}(I_+),T_{\eps}(J_+)},\mathbb{C}),$ which is bounded by
      Lemma \ref{lemma:mod-reldist}.\\
$(a2)$ If $I_+\Subset (1,1+i\cdot\infty)$ and $J_+\Subset(1,\infty)$
    then $T_{\eps}(J_+)=J_+$ while $T_{\eps}(I_+)$ is eventually contained
    in an interval $(1,1+\delta i)$ for every $\d>0$. Therefore
    $\dist(T_{\eps}(I_+),T_{\eps}(J_+))\to \dist(\{0\}, J_+) >0,$ and
$\diam T_{\eps}(I_+)\to0$ as $\eps\to 0.$ Thus,
$\D(T_{\eps}(I_+),T_{\eps}(J_+))\to\infty$ and $\m\G_{I_+,J_+}^{\eps}$ is
bounded by Lemma \ref{lemma:mod-reldist}.\\

If $1\in I_+ \cup J_+$ then there are two more cases (we are assuming that
$I_+$ is located to the left of $J_+$ when looking from inside $C$):\\
$(a3)$ If $1\in I_+$ while $J_+\subset(1,\infty)$ then
$\D(T_{\eps}(I_+),T_{\eps}(J_+))\to \D(I_+\cap\mathbb{R},J_+) >0,$ as
$\eps\to0$ and therefore $\m \G^{\eps}_{I_+,J_+}$ is bounded.\\
$(a4)$ If $1\in J_+$ then we may assume that there are reals $0<c<a<b<\infty$
and $d>0$ such that $I_+=(1+ia,1+ib)$ and $J_+=[1,1+ic)\cup[1,d)$. Then
\begin{align*}
  \D(T_{\eps}(I_+),T_{\eps}(J_+)) = \frac{\eps (a - c)}{\eps c} =\frac{a}{c}-1>0
\end{align*}
and $\m\G_{I_+,J_+}^{\eps}$ is bounded. The same arguments show that
$\m\G_{I_-,J_-}$ is also bounded in this case.\\

\textsc{Case} $(b):$ If $i\in\phi(I)$ and $1\notin J_+$ (we also assume
$1\notin\partial J_+$) then either $J_+\subset C_+\setminus\mathbb{R}$ or
$J_+\Subset(1,\infty)$. In the former case the proof follows the same lines
as in Case $(a1)$ above. Therefore we assume
$$ i\in\phi(I) \mbox{ and } J_+\Subset(1,\infty).$$
Since $\m\G_{I_+,J_+}^{\eps}\leq
      \m\G_{I_+\cap\mathbb{R},J_+}^{\eps}+
      \m\G_{I_+\setminus\mathbb{R},J_+}^{\eps}$
and
$$\m\G_{I_+\cap\mathbb{R},J_+}^{\eps} =
\m\G_{I_+\cap\mathbb{R},J_+}\leq \m\G_{I,J}<\infty,$$ we only need to show
that  $\m\G_{I_+\setminus\mathbb{R},J_+}^{\eps}$ is bounded. By
subadditivity,
\begin{align*}
\m\G_{I_+\setminus\mathbb{R},J_+}^{\eps}
&\leq
  \m\G_{[1,1+i\infty),J_+}^{\eps} =\m\G_{[1,1+i\infty),J_+}\\
  &\leq \m\G_{[1+i,1+i\infty),J_+}+\m\G_{[1,1+i),J_+}.
\end{align*}
Since $\m\G_{[1+i,1+i\infty),J_+}\leq 2$ (because
  $\G_{[1+i,1+i\infty),J_+}$ overflows the ``vertical family" in the
  rectangle $[-1,1]\times[0,1]$), and $\m\G_{[1,1+i),J_+}<\infty$ since
  $\D([1,1+i),J_+)>0$ (note that $\dist([1,1+i), J_+)>0$), it follows that
  $\m\G_{I_+\setminus\mathbb{R},J_+}^{\eps}$ is bounded.\\
%
%
%
%
%
%
%

\textsc{Case} $(c):$ If $I_+\Subset [1,1+i\infty)$ then just like in case
$(a4)$ above (with $b=\infty$), we have that $\m\G_{I_+,J_+\setminus
\mathbb{R}}^{\eps}$ is bounded. Therefore we only need to estimate
  $\m\G_{I_+,J_+\cap\mathbb{R}}^{\eps}$ and thus, we may assume
  $J_+\subset\mathbb{R}.$ In particular, without loss of generality  we assume
  that that there are reals $0<c<a<\infty$ and
  $1<d<\infty$ such that
  $$I_+=(1+ia,1+i\infty) \mbox{ and } J_+=[1,d).$$
%
Therefore
\begin{align*}
 \m\G^{\eps}_{I_+,J_+}
 &\leq \m\G^{\eps}_{I_+\cap I_0 ,J_+\cap J_0} +\m\G^{\eps}_{I_+\cap I_0 ,J_+\setminus J_0} \\
 & + \m\G^{\eps}_{I_+\setminus I_0,J_+\cap J_0}+\m\G^{\eps}_{I_+\setminus I_0,J_+\setminus J_0}.
\end{align*}

 - Now, if $J_+\subset J_0$ then $\m\G^{\eps}_{I_+\cap I_0 ,J_+\setminus
J_0}=0$. However, if $J_+\supset J_0$, then for every
$I'\subset[1,1+i\infty)$ we have
\begin{align}
\liminf_{\eps\to0}\D(T_{\eps}(I'), T_{\eps}(J_+\setminus J_0)) \geq \frac{1}{\diam(J_+\setminus J_0)}>0.
\end{align}
In particular $\m\G^{\eps}_{I_+\cap I_0 ,J_+\setminus J_0}$ and
$\m\G^{\eps}_{I_+\setminus I_0,J_+\cap J_0}$ are both bounded as
$\eps\to0$.\\

- Note, also that $\m\G^{\eps}_{I_+\setminus I_0,J_+\cap
J_0}\leq\m\G^{\eps}_{I_+\setminus I_0,[1,\infty)},$
which is bounded, since
$$\D(T_{\eps}(I_+\setminus
I_0),T_{\eps}([1,\infty)))=
\D(I_+\setminus I_0,[1,\infty))>0.$$

Since, $\m\G^{\epsilon}_{I_+,J_+}\geq \m\G_{I_0,J_0}^{\epsilon}\to\infty$ it follows that
  \begin{equation}
    \lim_{\eps\to0} \frac{\m \G_{I_+,J_+}^{\eps}}{\m \G_{I_+\cap I_0,J_+\cap J_0}^{\eps}}=1.
  \end{equation}
Similarly, using the inequality
\begin{align*}
 \m\G^{\eps}_{I_0,J_0}
 &\leq \m\G^{\eps}_{I_+\cap I_0 ,J_+\cap J_0} +\m\G^{\eps}_{I_+\cap I_0 ,J_0\setminus J_+} \\
 & + \m\G^{\eps}_{I_0\setminus I_+,J_+\cap J_0}+\m\G^{\eps}_{I_0\setminus I_+,J_0\setminus J_+},
\end{align*}
we obtain
\begin{equation}
    \lim_{\eps\to0} \left(\frac{\m \G_{I_0,J_0}^{\eps}}{\m \G_{I_+\cap I_0,J_+\cap J_0}^{\eps}}\right)^{-1}=1.
  \end{equation}
Finally, combining the last two equalities we obtain (\ref{limit:chimney}).
\end{proof}



%
%
%

\end{document}